\documentclass[a4paper,12pt,reqno]{amsart}
\usepackage[foot]{amsaddr}
\usepackage[english]{babel}
\usepackage[utf8]{inputenc}
\usepackage{mathtools}
\usepackage{fixmath}
\usepackage{amssymb}
\usepackage{multirow}
\usepackage{alltt}
\usepackage{graphicx}
\usepackage{amsmath}
\usepackage{bbm}
\usepackage{soul}
\usepackage{siunitx}
\usepackage[T1]{fontenc}
\usepackage{textcomp}
\usepackage{fullpage}
\usepackage{comment}
\usepackage{lmodern}
\usepackage[usenames,dvipsnames]{xcolor}
\definecolor{darkblue}{RGB}{0,0,170}
\definecolor{darkerblue}{RGB}{0,20,120}
\definecolor{brickred}{RGB}{200,0,0}
\definecolor{tempcolor}{RGB}{200,0,0}
\usepackage{soul}
\usepackage{booktabs}

\usepackage[pdftitle={A Bayesian approach to inverse problems in spaces of measures}, hyperindex]{hyperref}
\usepackage[top=1.in, right=1.in, left=1.1in, bottom=1.1in]{geometry}
\hypersetup{colorlinks=true, linkcolor=darkblue, citecolor=brickred}
\makeatletter
\def\@seccntformat#1{%
	\protect\textup{%
		\protect\@secnumfont
		\expandafter\protect\csname format#1\endcsname % <--- added
		\csname the#1\endcsname
		\protect\@secnumpunct
	}%
}
\makeatother

\makeatletter
\def\th@plain{%
  \thm@notefont{}% same as heading font
  \slshape % body font
}
\def\th@definition{%
  \thm@notefont{}% same as heading font
  \normalfont % body font
}
\makeatother

\theoremstyle{plain}
\newtheorem{definition}{Definition}[section]
\newtheorem{example}[definition]{Example}
\newtheorem{theorem}[definition]{Theorem}
\newtheorem{proposition}[definition]{Proposition}

\newtheorem{remark}[definition]{Remark}

\newtheorem{corollary}[definition]{Corollary}
\newcounter{assum}
\newtheorem{assumption}[assum]{Assumption}
\addto\captionsenglish{}

\newcommand{\R}{\mathbb{R}}
\newcommand{\N}{\mathbb{N}}
\newcommand{\C}{\mathcal{C}}
\newcommand{\Cc}{\mathcal{C}}

\newcommand{\A}{\mathcal{A}}

\newcommand{\M}{\mathcal{M}}
\newcommand{\G}{\mathcal{G}}

\newcommand{\Mm}{\mathcal{M}}
\newcommand{\B}{\mathcal{B}}
\renewcommand{\P}{\mathbb{P}}
\newcommand{\E}{\mathbb{E}}

\DeclarePairedDelimiter{\norm}{\lVert}{\rVert}

\DeclarePairedDelimiter{\pair}{\langle}{\rangle}
\DeclarePairedDelimiter{\dual}{\langle}{\rangle}
\DeclarePairedDelimiter{\inner}{(}{)}
\newcommand{\de}{\mathop{}\!\mathrm{d}}
\renewcommand{\d}{\mathrm{d}}
\DeclareMathOperator{\supp}{supp}

\DeclareMathOperator{\post}{post}
\DeclareMathOperator{\pr}{pr}

\DeclareMathOperator{\var}{var}

\DeclareMathOperator{\Hell}{Hell}

\newcommand{\miu}{u}
\newcommand{\bld}[1]{\boldsymbol{#1}}

\renewcommand{\epsilon}{\varepsilon}

\usepackage{enumitem}
\setenumerate{nolistsep}
\usepackage{mdframed}
\usepackage{todonotes}

\newcounter{constantC}
\newcommand{\newconstant}[1]{\refstepcounter{constantC}\label{#1}}
\newcommand{\useconstant}[1]{C_{\ref{#1}}}
\newcommand{\defconstant}[1]{ \newconstant{C_#1}\expandafter\newcommand\csname C#1\endcsname{\useconstant{C_#1}}}  %

\newcounter{constantc}
\newcommand{\newconstantc}[1]{\refstepcounter{constantc} \label{#1}}
\newcommand{\useconstantc}[1]{c_{\ref{#1}}}
\newcommand{\defconstantc}[1]{ \newconstantc{c_#1}\expandafter\newcommand\csname c#1\endcsname{\useconstantc{c_#1}}}

\newcommand{\Truong}[1]{{#1}}

\numberwithin{equation}{section}

\usepackage{lmodern}

\graphicspath{{plots/}}
\begin{document}

\title[A Bayesian approach to inverse problems in spaces of measures]{
A Bayesian approach to inverse problems \\ in spaces of measures
}
\author[P.-T. Huynh]{Phuoc-Truong Huynh}
\address{Institut f\"{u}r Mathematik, Alpen-Adria-Universit\"{a}t Klagenfurt,  9020 Klagenfurt, Austria.}
\email{phuoc.huynh (at) aau.at}

\begin{abstract} In this work, we develop a Bayesian framework for solving inverse problems in which the unknown parameter belongs to a space of Radon measures taking values in a separable Hilbert space. The inherent ill-posedness of such problems is addressed by introducing suitable measure-valued priors that encode prior information and promote desired sparsity properties of the parameter. Under appropriate assumptions on the forward operator and noise model, we establish the well-posedness of the Bayesian formulation by proving the existence, uniqueness, and stability of the posterior with respect to perturbations in the observed data. In addition, we also discuss computational strategies for approximating the posterior distribution. Finally, we present some examples that demonstrate the effectiveness of the proposed approach.

\bigskip

\noindent \textsc{Keywords.} Bayesian inverse problems, Radon measures, non-Gaussian prior.
	
\bigskip
	
\noindent \textsc{2020 Mathematics Subject Classification.} 35R30, 35Q62, 62F15, 62G35.

\end{abstract}

\maketitle
\setcounter{tocdepth}{1}
\tableofcontents

\section{Introduction}
In this work, we study an infinite-dimensional Bayesian framework for solving inverse problems, where the emphasis is placed on incorporating sparsity-promoting prior knowledge about the unknown signal. Such sparsity assumptions are motivated by the observation that in many real-world applications, the underlying physical parameters to be identified are localized or concentrated in space, which results naturally in sparse representations. 

To be precise, we consider an additive model for inverse problems of the form
\begin{equation}\label{eq:5_BIP}
	z = \mathcal{G}(u) + \xi,
\end{equation}
where $u$ is the unknown parameter, $z$ represents the vector of measurement data, $\G$ is the forward operator and $\xi$ is the measurement noise. We assume that the $u$ admits a \textit{sparse} structure, in the sense that it can be expressed as a finite sum of weighted Dirac measures supported on a domain $\Omega$. Specifically, $u$ is a discrete measure taking values in a separable Hilbert space $H$, given by
\begin{equation}\label{eq:groundtruth_form} u = \sum_{k=1}^{N_s} q_k \delta_{y_k}, \quad q_1, \ldots, q_{N_s} \in H, \quad y_1, y_2,\ldots, y_{N_s} \in \Omega, \end{equation}
where $N_s$ denotes the number of sources (typically small), $q_1, \ldots, q_{N_s}$ are the \textit{amplitudes} in $H$, and $y_1, \ldots, y_{N_s}$ are their locations. This type of signal has various applications, such as in Helmholtz source identification \cite{engel_hafemeyer_munch_schaden_2019}, optimal control problem of the wave equation with measure-valued control \cite{Trautmannseismic}, and optimal transport problems \cite{bredies_fanzon_2020}. Hence, the forward operator $\G$ maps from the parameter space $\M(\Omega, H)$, which is the space of Radon measures on $\Omega$ that take values in a Hilbert space $H$, to a Banach space $Y$. Detailed descriptions of the parameter space $\M(\Omega, H)$ will be given in Section~\ref{sec:5_vectormeasures}. Since the map $\G$ is typically not continuously invertible, this problem is, in general, ill-posed; thus, regularization methods are required.

In recent years, the Bayesian approach as a means of regularization has become more popular. We recall from \cite{stuart_2010, stuart_2017} that a Bayesian solution to \eqref{eq:5_BIP} is a probability measure $\mu^z_{\post}$, namely the posterior distribution of $u$ conditioned on measurement data $z$, as described by Bayes' rule \cite{stuart_2010}
\begin{align*}
\label{eq:bayesian_formula}
	\dfrac{\d \mu^z_{\post}}{d\mu_{\pr}} (u) = \dfrac{1}{Z(z)} \exp(-\Phi(u;z)), \text{ where } Z(z) = \int_{\M(\Omega, H)} \exp(-\Phi(u;z)) \d \mu_{\pr}(u),
\end{align*}
provided that all the quantities are well-defined. 
Here, $\mu_{\pr}$ is the \textit{prior measure} reflecting our belief about the parameter $u$, $\Psi(u;z)$ is the \textit{likelihood potential}, and $\mu^z_{\post}$ is the posterior measure representing the distribution of $u$ after incorporating the observed data $z$.
As is well known, the space $\M(\Omega, H)$ is a nonseparable Banach space, which is, in fact, the main challenge in studying the inverse problem \eqref{eq:5_BIP} within the Bayesian setting, since the well-posedness theory in \cite{stuart_2010} cannot be directly applicable. In addition, a suitable prior measure $\mu_{\pr}$ on $\M(\Omega, H)$ should be introduced to establish the well-posedness of the Bayesian inverse problem, as well as to ensure computationally feasible sampling of the solutions.

Since the pioneering work by Stuart \cite{stuart_2010}, the theory of Bayesian inverse problems has been extensively developed. The well-posedness of Bayesian inverse problems has been studied in several settings, such as in quasi-Banach spaces \cite{sullivan_2017}, in non-separable Banach spaces with unconditional Schauder bases \cite{hosseini_2017}. The stability of the map $z \mapsto \mu_{\post}^z$ has been considered in various distances on the space of probability measures, for instance, the Hellinger distance \cite{stuart_2010}, the total variation distance \cite{hosseini_2017}. Notably, the work \cite{latz_2020} introduced a general concept for the well-posedness of Bayesian inverse problems, under a mild assumption on the measurability of the parameter space and the likelihood function $L(z|u)$. In addition, the study of prior measures is of particular interest for characterizing prior beliefs about the ground truth. For instance, Besov priors \cite{dashti_harris_stuart_2012, agapio_burger_dashti_helin_2018} promote sparsity in the representation of functions, while convex and heavy-tailed priors \cite{hosseini_2017, hosseini_nigam_2017} offer alternative structural properties.

From a practical point of view, Bayesian approaches have been employed in various applications, such as Helmholtz source identification with Dirac sources \cite{engel_hafemeyer_munch_schaden_2019}, scalar conservation laws \cite{dashti_duong_2023}, and geophysics \cite{bui_ghatttas_martin_stadler_2013}. We remark that the application of Bayesian inverse problems most closely related to our work is found in \cite{engel_hafemeyer_munch_schaden_2019}, where a prior measure is defined on the space $\ell^1$ of summable sequences. While the authors' main result on well-posedness is mathematically justified and presents a notable application to the sound source localization problem within the Bayesian framework, we believe that this manuscript offers a more intuitive approach to Bayesian inverse problems that incorporates the sparsity assumption of the underlying signal.

\subsection{Contributions} The main contributions of this manuscript are the following:
\begin{enumerate}
	\item[(1)] We introduce a suitable prior measure $\mu_{\pr}$ on the space of measures $\M(\Omega, H)$. This prior is characterized in terms of point processes, see, for instance, \cite{daley_vere_2003, reiss_1993}. While there is a rich theory on point processes for nonparametric Bayesian estimation (see \cite{broderick_2018} and the references therein), most existing work considers point processes with real, positive coefficients. We instead consider vector-valued coefficients and show that the resulting point process belongs to $\M(\Omega, H)$ $\mu_{\pr}$-almost surely.
	
    \item[(2)] Given the aforementioned priors, we study the well-posedness of the Bayesian inverse problem in this setting. Our approach follows that of \cite{latz_2020}, where the well-posedness of the Bayesian inverse problem \eqref{eq:5_BIP} is characterized under mild assumptions on the parameter space $\M(\Omega, H)$ and the likelihood function $L(z|u) = \exp(-\Psi(u;z))$. To this end, rather than treating $\M(\Omega, H)$ as a Banach space, we employ the weak* topology on $\M(\Omega, H)$, leveraging favorable properties of this topological structure to obtain suitable measurability properties on this space. Consequently, under the weak* continuity assumption on the forward operator, we establish its measurability, thereby making the theory in \cite{latz_2020} applicable.
\end{enumerate}
Finally, we present some examples that illustrate the applicability of our approach within this setting.

\subsection{Organization of the paper} The paper is organized as follows:
In Section~\ref{sec:5_vectormeasures}, we recall some topological and measure-theoretic properties of the space $\M(\Omega, H)$. Based on these properties, we define a suitable prior on the space $\M(\Omega, H)$ in Section~\ref{sec:5_priordistribution}. In Section~\ref{sec:5_wellposedness}, we establish a well-posedness result for the Bayesian inverse problem \eqref{eq:5_BIP}. Finally, we present some examples that illustrate the applicability of the developed theory.

\section{Vector-valued measures}\label{sec:5_vectormeasures}

Let us introduce some notation that will be used in the sequel.  First, we denote by $\Omega \subset \R^d$, $d \ge 1$ the closure of a bounded domain and denote by $H$ a separable Hilbert space equipped with the inner product $(\cdot, \cdot)_{H}$. Also, we consider two types of measures that are crucial for the analysis: The first type includes elements of \( \M(\Omega, H) \), which are (Radon) \( H \)-valued measures on \( \Omega \), and the second type consists of probability measures on \( \M(\Omega, H) \). Elements of the first type will be denoted by \( u, v, \) etc., while those of the second type will be denoted by \( \mu, \nu, \) etc., for clarity.

\subsection{Topological properties of \( \M(\Omega,  H)\)} We first introduce vector measures on $\Omega$. An $H$-valued measure on $\Omega$ is a countably additive mapping $u\colon \mathcal{B}(\Omega) \to H$, where $\mathcal{B}(\Omega)$ denotes the Borel $\sigma-$algebra on $\Omega$. For every $H$-valued measure $u$, we define the associated total variation measure $|u|: \mathcal{B}(\Omega) \to \R^+ \cup \{ + \infty \}$ by
\begin{align*}
	|u|(B) := \sup \left\{ \sum_{n=1}^{\infty} \norm{u(B_n)}_{H} : \{B_n\}_{n \in \mathbb{N}} \subset \mathcal{B}(\Omega) \text{ is a disjoint partition of }B\right\},
\end{align*}
for every $B \in \mathcal{B}(\Omega)$. Hence, the space $\M(\Omega, H)$ is the space of $H$-valued measures on $\Omega$ with finite total variations, namely
\begin{align*}
	\M(\Omega, H) := \left\{ u \text{ is an }H\text{-valued measure on } \Omega : |u|(\Omega) < \infty \right\},
\end{align*}
which is a Banach space equipped with the norm
\[ \norm{u}_{\M(\Omega, H)} := |u|(\Omega) = \int_{\Omega} \d |u|. \] 

The support of the vector measure $u$, defined in the usual way, satisfies $\supp u = \supp |u|$. Hence, for a discrete measure
\begin{align*}
	u = \sum_{k=1}^{N_s} q_k \delta_{y_k}, \quad q_k \in H, y_k \in \Omega, \text{ for all }k = 1,\ldots, N_s,
\end{align*}
one has
\begin{align*}
	\supp u = \{ y_1,\ldots, y_{N_s}\} \quad \text{and}\quad \norm{u}_{\M(\Omega, H)} = \sum_{k=1}^{N_s} \norm{q_k}_{H}.
\end{align*}
Note that for $H = \R$, we obtain the classical space of real-valued measures $\M(\Omega, \R) \equiv \M(\Omega)$, see for instance \cite{bredies_pikkarainen_2013}, while $H = \mathbb{C}^n$ gives $\M(\Omega, \mathbb{C}^n)$ corresponding to the space of complex vector-valued measures in \cite{pieper_tang_trautmann_walter_2020}.

Next, we can see that every $u \in \M(\Omega, H)$ is absolutely continuous with respect to $|u|$, meaning that if $B \in \mathcal{B}(\Omega)$ and $|u|(B) = 0$, then $u(B) = 0_{H}$. Hence, by the Radon-Nikodym Theorem for vector measures, see \cite[Corollary 12.4.2]{lang_1983}, there exists a function $u': \Omega \to H$ such that \[ \norm{u'}_H \in L^{\infty}(\Omega, \d u), \text{ with } \norm{u'(x)}_H = 1 \text{ for } |u|-\text{almost every }x \in \Omega
\]
and $u$ can be decomposed into
\[
u(B) = \int_{B} \d u = \int_{\Truong{B}} u' \d |u|, \quad \text{ for all }B \in \B(\Omega).
\]
Equivalently, $\M(\Omega, H)$ can be characterized as the dual of $\C(\Omega, H)$, where $\C(\Omega,H)$ is the space of bounded continuous functions on $\Omega$ taking values in $H$. By Singer's representation theorem (see, e.g., \cite{hensgen_1996}), the duality pairing between $\M(\Omega,H)$ and $\C(\Omega,H)$ is defined by
\[ \pair{u, f}_{\M(\Omega,H), \C(\Omega,H)} = \int_{\Omega} f \d u = \int_{\Omega} \left(f(x), u'(x)\right)_H \d |u|(x).\]
By a slight abuse of notation, we will simply write $\pair{u, f}$ to denote the dual pairing between $u \in \M(\Omega, H)$ and $f \in \C(\Omega, H)$, unless otherwise stated. Hence, the norm on $\M(\Omega, H)$ is also characterized by the dual norm
\begin{align*}
\norm{u}_{\Mm} &= \sup \left\{ \pair{u, f} : f \in \C(\Omega, H), \,\sup_{x \in \Omega} \norm{f(x)}_H \le 1 \right\} \\
&= \sup \Big \{ \int_{\Omega} (f(x), u'(x))_H \d |u|(x): f \in \C(\Omega, H), \, \sup_{x \in \Omega} \norm{f(x)}_{H} \le 1 \Big\}.
\end{align*}
Since $\C(\Omega, H)$ is nonreflexive, the structure of $\M(\Omega, H)$ is rather complicated; in particular, $\M(\Omega, H)$ itself is both nonreflexive and nonseparable. Nevertheless, when equipped with the weak* topology $w^*$ (i.e. the coarsest topology \Truong{such} that every map $f \mapsto \pair{u, f}$ is continuous), the space $\M(\Omega, H)$ becomes a locally convex Hausdorff space whose dual is $\C(\Omega, H)$. This property characterizes the weak* convergence on $\M(\Omega, H)$, namely a sequence of measures $\{ u_k \}_{k \in \mathbb{N}}$ converges to a limit $u \in \M(\Omega, H)$ if 
\begin{align*}
	\pair{u_k, f} \to \pair{u, f} \text{ as } k \to \infty,\quad \text{ for all } f \in \C(\Omega, H).
\end{align*}

Equipped with the weak* topology $w^*$, it can be seen that $(\M(\Omega, H), w^*)$ is a complete locally convex Hausdorff topological vector space. In addition, $\M(\Omega, H)$ is a Souslin space \cite{saint_1978}, i.e., an image of a separable, completely metrizable space under a continuous map. For more details on the topological properties of the space of vector measures, we refer to \cite[Section 12.3]{lang_1983} and \cite{pieper_2015}.

\subsection{Measurability on $\M(\Omega, H)$} In the following, we will recall several measure-theoretic properties of the space $\M(\Omega, H)$. Many of these results are classical and can be implied from properties in standard references such as \cite{diestel_1977, lang_1983, Bogachev_2007}. However, since precise references for some statements are not readily available, we also include their proofs for completeness.

 In order to introduce prior measures on $\M(\Omega, H)$,  an appropriate Borel $\sigma$-algebra on $\M(\Omega, H)$ should be introduced.  The following $\sigma$-algebras on $\M(\Omega, H)$ are natural to consider:
\begin{enumerate}
	\item[(1)] The strong Borel $\sigma$-algebra, denoted by $\mathcal{B}$, generated by the strong topology (the norm topology) on $\M(\Omega, H)$.
	\item[(2)] The weak* Borel $\sigma$-algebra, denoted by $\B_{w^*}$, generated by the weak* topology on $\M(\Omega, H)$.
\end{enumerate}

In separable Banach spaces, the two $\sigma$-algebras, in fact, coincide, see \cite[Section A.2.2]{stuart_2017}. However, since $\M(\Omega, H)$ is not separable, this property does not hold in $\M(\Omega, H)$; see Remark~\ref{rem:strongvsweak}. In the following, we show that the $\sigma$-algebra $\B_{w^*}$ is appropriate for our analysis. We begin by showing that $\B_{w^*}$ can also be characterized by linear functionals on $\M(\Omega, H)$.

\begin{proposition}\label{prop:equivalence_sigma} The $\sigma$-algebra $\B_{w^*}$ coincides with the $\sigma$-algebra generated by $\C(\Omega, H)$, which corresponds to a collection of linear functionals on $\M(\Omega, H)$ via $u \mapsto \pair{u, f}$, $f \in \C(\Omega, H)$.
\end{proposition}

\begin{proof} Denote by $\widetilde{\sigma}$ the $\sigma$-algebra generated by the given set of linear functionals on $\M(\Omega, H)$. For every $f \in \C(\Omega, H)$, the linear functional $u \mapsto \pair{u, f}$ is continuous in $w^*$, and therefore $\B_{w^*}$-measurable. Consequently, $\widetilde{\sigma} \subset \B_{w^*}$. 

Conversely, we know that for every $f \in \C(\Omega, H)$ and $r > 0$, the set \[\{ u \in \M(\Omega, H): \pair{u, f} < r\}\] is $\widetilde{\sigma}$-measurable. Hence, let $u_0 \in X$ be given. For a finite set $\{ f_1, f_2, \ldots, f_k\} \subset X_*$ and $r > 0$, the set
\begin{align*}
	V = V(u_0; f_1,f_2,\ldots, f_k; r) := \left\{ u \in X: | \pair{u - u_0, f_i}| < r, \,\,\text{ for all } i = 1,\ldots, k\right\},
\end{align*}
defining a neighborhood of $u_0$ for the weak* topology, is also $\sigma_1$-measurable. As the sets of this type form a basis of the weak* topology, this implies $\B_{w^*} \subset \widetilde{\sigma}$. In conclusion, we have $\B_{w^*} = \widetilde{\sigma}$ and the proof is complete.
\end{proof}
%
%\begin{remark} 
%	
%On a separable Hilbert space $X$, the two Borel $\sigma$-algebras generated by the strong topology and the weak topology, coincide. Indeed, denote by $\B$ and $\B_{w}$ the strong and weak Borel $\sigma$-algebra on $X$, respectively.
%We know that every closed ball $\overline{B_r}(x), x \in X, R > 0$ is weakly closed, and hence belongs to $\B_{w}$. Hence, every open ball $B_r(x)$ can be written by $B_r(x) = \cup_{n=1}^{\infty} \overline{B_{r - 1/n}}(x)$, and therefore $B_r(x) \in \B_{w}$ for all $x \in X$, $R > 0$. As $\B$ is generated by all the balls of this form, we know that $\B \subset \B_{w}$. The reverse inclusion follows directly from the fact that the weak topology is coarser than strong topology. 
%
%In fact, the same conclusion holds for separable Banach spaces. Thus, the choice of Borel $\sigma-$algebra does not affect the resulting measure-theoretic setting for Bayesian inverse problems in separable Banach spaces; cf. \cite[Section A.2.2]{stuart_2017}. 
%\end{remark}

With the given $\sigma$-algebra on $\M(\Omega, H)$, we are ready to introduce several measurable maps from and to $\M(\Omega, H)$ which will be useful in the sequel. In the following, we denote by $(W, \A)$ a measurable space. We first introduce measurability properties of maps from $W$ to $\M(\Omega, H)$. For details, we refer to \cite[Chapter II]{diestel_1977}, \cite{casas_clason_kunish_2013} and the references therein.

\begin{definition}Consider a map $U: W \to \M(\Omega, H)$.
	\begin{enumerate}
		\item[(1)] $U$ is \textit{weakly* measurable} if for every $f \in \C(\Omega, H)$, the map 
		\[ w \mapsto \pair{U(w), f},\quad w \in W \]
		is measurable from $W$ to $\R$ with its Borel $\sigma$-algebra.
		\item[(2)] $U$ is \textit{$\B_{w^*}$-measurable} if it is a measurable map to $(\M(\Omega, H), \B_{w^*})$.
	\end{enumerate}
\end{definition}
In fact, the two notions of measurability are equivalent, as follows directly from the definitions of the two coinciding $\sigma$-algebras on $\M(\Omega, H)$.

\begin{proposition}\label{prop:equiv_measurability} A map $U : W \to \M(\Omega,H)$ is weakly* measurable if and only if it is $\B_{w^*}$-measurable.
\end{proposition}

\begin{proof} Assume that $U$ is $\B_{w^*}$-measurable. For every $f \in \C(\Omega, H)$ and $u \in \M(\Omega, H)$, the map $u \mapsto \pair{u, f}$ is measurable. Hence, as a composition of measurable maps, the map $w \mapsto \pair{U(w), f}$ is measurable. Hence, $U$ is weakly* measurable.
	
Conversely, assume that $U$ is weakly* measurable. Hence, for every $f \in \C(\Omega, H)$ and $r > 0$, the set
\begin{align*}
	W(f; r) := \left\{ w \in W: |\pair{U(w), f}| < r\right\}
\end{align*}
is measurable. Hence, let $u_0 \in \M(\Omega, H)$. For every $f_1,\ldots, f_k \in C(\Omega, H)$ and $r > 0$, the set
\begin{align*}
	W(u_0; f_1,\ldots, f_k; r) := \left\{ w \in W: \left|\pair{U(w) - u_0, f_i} \right| < r   \,\,\text{ for all } i = 1,\ldots, k\right\}
\end{align*}
is measurable. This set is in fact $U^{-1}\left(V(u_0;f_1,\ldots, f_k; r)\right)$, and thus implies the measurability of $U$ as a map from $W$ to $\M(\Omega, H)$. The proof is complete.
\end{proof}

%{\color{brickred}The following argument is true in the case $H = \R$ (see \cite{dubins_freedman_1964}). However, I am trying to find a reference for those in the case $H$ is a Hilbert space (or at least $H$ is an $n-$dimensional space with $n \ge 2$).}

Hence, in the following, by a measurable map $U: W \to \M(\Omega, H)$ or $U: \M(\Omega, H) \to W$, we mean it is measurable with respect to $\B_{w^*}$ (equivalently, it is also weakly* measurable).

\begin{corollary} Let $Q: W \to H$ and $Y: W \to \Omega$ be measurable maps. Then the map $U : W \to \M(\Omega, H)$ given by $U(w) = Q(w) \delta_{Y(w)}$ is measurable.
\end{corollary}

\begin{proof} For every $f \in \C(\Omega, H)$, we consider the map $\varphi_f : W \to \R$ given by
	\begin{align*}
		\varphi_f(w)	 = \dual{ U(w), f}  = (Q(w), f(Y(w)))_H, \quad \text{ for all } w \in W.
	\end{align*}
	By the measurability of $Q$ and $Y$, and the continuity of the inner product $(\cdot, \cdot)_H$, it follows that $\varphi_f$ is measurable for every $f \in \C(\Omega, H)$. Hence, by Proposition~\ref{prop:equiv_measurability}, the map $U$ is measurable.
\end{proof}

We next study some measurable maps mapping from $\M(\Omega, H)$.

\begin{proposition}Let $K$ be a separable Hilbert space. Then a map $\G: \M(\Omega, H) \to K$is measurable if and only if the map $u \mapsto (\G(u), g)_H$ is measurable for every $g$ in $K$.
\end{proposition}

\begin{proof}We proceed as in the proof of Proposition~\ref{prop:equivalence_sigma}. If $\G$ is measurable, then the map $u \mapsto (\G(u), g)_K$ is measurable as it is the composition of a measurable map and a continuous map. Conversely, assume that the map $u \mapsto (\G(u), g)_K$ is measurable for every $g$. Hence the set $\G^{-1}(V(g_0, r))$ is measurable for every $g_0 \in K$ and $r > 0$, where
\begin{align*}
V(g_0, r) := \left\{ g \in K: (g, g_0)_K < r\right\} 
\end{align*}
is measurable. Since these sets form the Borel $\sigma$-algebra on $K$, we conclude that $\G$ is measurable. The proof is complete.
\end{proof}
	
\begin{corollary}Let $K$ be a Hilbert space and $\mathcal{G}: \M(\Omega, H) \to K$ be a weak*-to-weak continuous map. Then $\mathcal{G}$ is measurable.
\end{corollary}
This follows from the fact that the map $u \mapsto (\G(u), g)_K$ is weak* continuous for every $g \in K$ and therefore measurable.

\begin{proposition}The norm map $\norm{\cdot}: \M(\Omega, H) \to \R$,  $u \mapsto \norm{u}_{\M(\Omega, H)}$ is measurable.
\end{proposition}

\begin{proof}First, notice that the space $\C(\Omega,H)$ is separable since $\Omega$ is compact. Denote by $S$ a countable subset of $C(\Omega,H)$ such that $\sup_{ x \in \Omega} \norm{f(x)}_H \le 1$ for all $f \in S$ and $S$ is dense in the unit ball of $C(\Omega,H)$. Since for every $f \in S$,  every map $u \mapsto \pair{u, f}$, is measurable by the definition, the map $u \mapsto \sup_{f \in B} \pair{u, f}$ is also measurable.  By the definition of the dual norm, one has
	\[ \norm{u}_{\M(\Omega,H)} = \sup_{\norm{f} \le 1} \pair{u, f} = \sup_{f \in S} \pair{u, f}.  \] 
	Since the supremum of a countable set of measurable functions is also measurable, we conclude that the map $u \mapsto \norm{u}_{\M(\Omega,H)}$ is measurable. The proof is complete.
\end{proof}

\begin{remark}\label{rem:strongvsweak} While the norm map is measurable, we remark that there exists a set that is measurable in $\B$ but not in $\B_{w^*}$. Indeed, we simply consider the space of real-valued Radon measures $\M(\Omega)$ and let $E$ be a non-Borel measurable {\Truong subset} of $\Omega$. On the space $\Omega$, we consider the set $\M_E := \left\{ \delta_x : x \in E\right\}$. Since for every $\delta_{x_1},\delta_{x_2} \in \M_E$, one has
\[ \norm{\delta_{x_1} - \delta_{x_2}}_{\M} = 2,\quad \forall x_1 \neq x_2,\] the set includes all of its limit points and \Truong{is} therefore closed in the strong topology. In particular, it is measurable in $\B$.
On the other hand, the map $\varphi: \Omega \to \M(\Omega)$, $\varphi(x) = \delta_x$ is weakly measurable by Proposition~\ref{prop:equiv_measurability}. Hence, the set $\M_E$ is not weakly measurable, otherwise $\varphi^{-1}(\M_E) = E$ would be measurable, which is a contradiction. Hence, we conclude that the two $\sigma$-algebras are not equivalent on $\M(\Omega)$. 

\end{remark}

For simplicity, and when no confusion arises, we will also write $\Mm$ and $\Cc$ to denote $\M(\Omega, H)$ and $\C(\Omega, H)$, respectively. 

\section{Prior distribution on the space of measures} \label{sec:5_priordistribution}

Following the discussion in the previous section, we are now able to define prior distributions,  or more precisely prior measures, on the space $\M(\Omega,  H)$, which represents our initial beliefs about the model parameters before observing any data. In the following, we explore several examples of prior measures, illustrating their properties and the motivations behind their choices in different modeling scenarios.

\subsection{Random measures} Central objects in defining prior measures on the space of measures are so-called random measures. More precisely, consider a probability space $(\Theta, \mathcal{F}, \mathbb{P})$. A random measure on $\M(\Omega, H)$ is a (weakly*) measurable map $U\colon \Theta \to \M(\Omega, H)$. This induces a Borel probability measure $\mu_{\pr}$ on $\M(\Omega, H)$ given by
\begin{align}\label{eq:prior_measure}
	\mu_{\pr}(E) := \P(U(\omega) \in E), \quad E \in \B_{w^*}.
\end{align}
In addition, since $\M(\Omega, H)$ is a Souslin space \cite{saint_1978}, every Borel measure is Radon \cite[Theorem 8.6.13]{cohn_2013}, meaning that every finite Borel measure on it is inner regular. Furthermore, one has the following characterization of Borel probability measures on $\M(\Omega, H)$:
\begin{proposition}[{\cite[Theorem 7.4.3]{Bogachev_2007}}] Every Borel measure $\mu$ on $\M(\Omega, H)$ is Radon and is concentrated on a countable union of metrizable compact sets.  In addition, for every Borel set $B$ and every $\varepsilon > 0$,  there exists a metrizable compact set $K_{\varepsilon} \subset B$ such that $|\mu|(B \backslash K_{\varepsilon}) < \varepsilon$.
\end{proposition}

Nevertheless, sampling a general random measure is nontrivial due to the non-separability of the space. This presents a challenging task in practical computation. To address this issue, it is necessary to consider a subset of random measures that not only ensures the well-posedness of the Bayesian inverse problem but also enables efficient sampling and numerical implementation. This motivates the development of structured priors or parametrizations that restrict the space of random measures to a computationally tractable class, while still capturing the essential features of the underlying inverse problem. As we have seen, since we are interested in parameters of the form \eqref{eq:groundtruth_form}, we consider the so-called class of \textit{point processes} which is defined as follows:

\begin{definition}
	Let $K$ be a random variable on $\mathbb{N}$, $\{Y_k\}_{k \in \mathbb{N}}$ a sequence of random variables $Y_k\colon\Theta \to \Omega$, and $\{Q_k\}_{k \in \mathbb{N}}$ a sequence of i.i.d. $H$-valued random variables. We consider the point process of the form
	\begin{equation}\label{eq:radon_random_measure}
		U \sim \sum_{k=1}^{K} \gamma_k Q_k \delta_{Y_k},
	\end{equation}
	where $\{\gamma_k\}_{k \in \N}$ is a fixed sequence of positive coefficients that decay sufficiently fast. Here, certain conditions on $\{\gamma_k\}_{k \in \mathbb{N}}$ and $\{ Q_k\}_{k \in \mathbb{N}}$ are needed to ensure that the measure $\mu_{\pr}$, defined as the distribution of $U$ given in \eqref{eq:prior_measure}, is a well-defined measure on $\M(\Omega, H)$.
\end{definition}

 We remark that here we do not assume that $K$ is independent of $\{(Y_k, Q_k)\}_{k \in \mathbb{N}}$. Without any confusion, we also write $u$ to denote $U$. The expression in \eqref{eq:radon_random_measure} is well-defined as a random measure on $\M(\Omega, H)$ according to the following result.

\begin{proposition}Let $u$ be given in \eqref{eq:radon_random_measure}.
\begin{enumerate}
	\item[(1)] If $K < \infty$ almost surely, then \eqref{eq:radon_random_measure} is a well-defined random measure on $\M(\Omega, H)$ for every sequence $\{\gamma_k\}_{k \in \mathbb{N}}$.
	\item[(2)] If $K = \infty$ almost surely, then \eqref{eq:radon_random_measure} is a well-defined random measure on $\M(\Omega, H)$ if $\{ |\gamma_k|^2\}_{k \in \mathbb{N}} \in \ell^p$ and $\{\var \norm{Q_k}_H \}_{k \in \mathbb{N}} \in \ell^q$ with $1/p + 1/q = 1$.
\end{enumerate}
\end{proposition}

\begin{proof}We adapt the proof in \cite{hosseini_2017}. First, assume that $K < \infty $ almost surely. The map 
\[ \omega \mapsto u (\omega) = \sum_{k=1}^{K(\omega)} \gamma_k Q_k(\omega) \delta_{Y_k(\omega)},\]
is measurable, since for every $n \in \mathbb{N}$ and $E \in \B_{w^*}$, the set
\begin{align*}
	W_n := \left\{ \omega \in \Theta: K(\omega) = n \text{ and } u(\omega) = \sum_{k=1}^n \gamma_k Q_k(\omega)\delta_{Y_k(\omega)} \in E \right\}
\end{align*}
is measurable. Hence, the set $W := \cup_{n=1}^{\infty} W_n = u^{-1}(E)$ is measurable. In addition, there holds
\begin{align*}
	\norm{u}_{\M(\Omega, H)}  \le \sum_{k=1}^{K(\omega)} |\gamma_k| \norm{Q_k(\omega)}_{H}.
\end{align*}
Since $\P(\omega: K(\omega) < \infty) = 1$, we conclude that $\norm{u}_{\M(\Omega, H)} < \infty$ almost surely.
	
Next, we assume that $K = \infty$. In this case, denote 
\begin{align*}
	u_n := \sum_{k=1}^n \gamma_k Q_k \delta_{Y_k}.
\end{align*}
It can be seen that $u_n$ is $H$-valued random measure. In addition, one has
\begin{align*}
	\norm{u_n}_{\M(\Omega, H)} \le \sum_{k=1}^n |\gamma_k| \norm{Q_k}_{H} :=  v_n.
\end{align*}
We prove that the sequence $\{v_n\}_{n \in \mathbb{N}}$ is bounded almost surely. Indeed, by H\"{o}lder's inequality, we have 
\begin{align*}
	\sum_{k=1}^{\infty} \var \norm{\gamma_k Q_k}_{H} 
	&= \sum_{k=1}^{\infty}  |\gamma_k|^2 \var \norm{Q_k}_H  \\
	& \le \norm{ \{ |\gamma_k|^2\}_{k \in \mathbb{N}}}_{\ell^p} \norm{ \{\var \norm{Q_k}_H \}}_{\ell^q} < \infty.	
\end{align*}
Hence, by Kolmogorov's Theorem, we have $\sum_{k=1}^{\infty}  |\gamma_k| \norm{Q_k}_{H} < \infty$ almost surely. Hence, $\{v_n\}_{n \in \mathbb{N}}$ is bounded almost surely. Finally, since $u_n \rightharpoonup^* u$ almost surely, we have
\[ \norm{u}_{\M(\Omega, H)} \le \liminf_{n \to \infty} \norm{u_n}_{\M(\Omega, H)} \le \liminf_{n \to \infty} v_n < \infty.\]
The proof is complete.
\end{proof}

We provide some examples of random measures satisfying the assumptions in Proposition \ref{prop:randommeasure}. 
\begin{example}[(Poisson point process)]
	\label{ex:poisson}	
Let $\mathcal{Q}$ be a probability measure on $H$, and $G$ be a density function on $\Omega$. We assume that $K$ follows the Poisson distribution $\text{Pois}(\gamma)$, i.e.,
\begin{align*}
		P(K(\omega) = n) = \frac{\gamma^n \exp(-\gamma)}{n!}, \quad n = 0,1,\ldots
\end{align*}
Hence $u = \sum_{k=1}^{K} Q_k\delta_{Y_k}$ defines a random variable taking values in $\Mm(\Omega, H)$. This random measure is closely related to Poisson point processes on $\Omega$, see for instance, \cite{daley_vere_2008}.
Typically, the intensity $\gamma$ is specified in terms of a measure $\lambda$ defined on $\Omega$ or on $\Omega \times H$, known as the rate measure. If $\lambda$ is given in terms of $Q$ and $Y$, that is,
\begin{align*}
	\d \lambda = \nu (\d \mathcal{Q}) \cdot \varphi(\d G),
\end{align*}
then the random variable $K$ is not indepedent of $Y_k$ or $Q_k$.
\end{example}

Next, we show that $u$ has a finite first moment under certain conditions.

\begin{proposition}\label{prop:randommeasure} Let $u$ be defined in \eqref{eq:radon_random_measure} with \Truong{$\gamma_k = 1$ for all $k \in \mathbb{N}$}.  Assume that $\mathbb{E}[K] < \infty$ and $\sup_{k \in \mathbb{N}} \mathbb{E}[\norm{Q_k}_H] < \infty$. If $K$ and each $Q_k$ are independent, for every $k \in \mathbb{N}$, then $\mathbb{E}_{\mu_{\pr}}[\norm{u}_{\Mm}] < \infty$. In addition, if $\{Q_k\}_{k \in \N} \sim \mathcal{Q}$ is a sequence of independent and identically distributed (i.i.d.) $H$-valued random variables, then 
\begin{align}\label{eq:finite_moment}
	\mathbb{E}_{\mu_{\pr}}[\norm{u}_{\Mm}] = \mathbb{E}[K]\cdot \mathbb{E}[\norm{Q_1}_H].
\end{align}
\end{proposition}

\begin{proof}Denote $M:= \sup_{k \in \N} \E [\norm{Q_k}_H] < \infty$. Since
	\begin{align*}
		\norm{u(\omega)}_{\Mm} = \sum_{k=1}^{K(\omega)} \norm{Q_k(\omega)}_H,\quad \omega \in \Omega,
	\end{align*}
one has
	\begin{align} \label{eq:finite_moment_2}
		\mathbb{E}\left[\norm{u}_{\Mm}\right] = \mathbb{E}\left[ \mathbb{E}[\sum_{k=1}^K \norm{Q_k}_H | K ]\right] \le  \mathbb{E} \left[ K M \right] = \mathbb{E}[K] M < \infty.
	\end{align}
Finally, \eqref{eq:finite_moment} follows from \eqref{eq:finite_moment_2} by using the i.i.d. properties of the sequence $\Truong{\{Q_k\}_{k \in \mathbb{N}}}$. This completes the proof.
\end{proof}

%Let $u = \sum_{k=1}^{\infty} \delta_{X_k}$ be a Poisson point process in $\Omega$ with intensity measure $\lambda$ such that $\lambda(\Omega) < \infty$. Let $\{Q_k\}_{k \in \mathbb{N}}$ be a sequence of i.i.d. random variables which are distributed by $\mathcal{N}(0,1)$. Then the measure
%\[ u = \sum_{k=1}^{\infty} Q_k \delta_{Y_k}\]
%	is a random measure satisfying Proposition \ref{prop:randommeasure}.

\subsection{Characterization of random measures}

In order to characterize a (probability) measure $\mu$ on a topological space, one might make use of its \textit{characteristic functional} $\widehat{\mu}$, see \cite[Section 7.13]{Bogachev_2007}. 

\begin{definition}Let $\mu$ be a measure on $\left(\M,  \sigma^*\right)$.  The characteristic functional of $\mu$ is the functional $\widehat{\mu}: \C \to \mathbb{C}$ defined by
	\begin{equation}
		\widehat{\mu}(f):= \int_{\Mm} \exp \left[i \pair{u,  f}\right] \d \mu(u),\quad \text{ for all } f \in \C.
	\end{equation}
\end{definition}
Since $(\Mm, \sigma^*)$ is a locally convex Hausdorff space, the characteristic functional uniquely determines the measure:
\begin{proposition}[{\cite[Proposition 7.13.4]{Bogachev_2007}}] Assume that $\mu_1$ and $\mu_2$ are Radon measures on $\M$.  Then $\mu_1 = \mu_2$ if and only if $\widehat{\mu_1} = \widehat{\mu_2}$.
\end{proposition}

As a typical example, the Poisson point process in Example~\ref{ex:poisson} is characterized by a measure of the form
\begin{align*}
	\widehat{\mu}(f) 
	&= \exp \left(\gamma \int_{\Omega}\int_{H}\exp(i \inner{q, f(y)}_H ) - 1) \de q \de y \right) \\
	&= \exp \left(\gamma \int_{\M} \exp(i \pair{f, u}) - 1 \right)\de \mu_{0}(u),
\end{align*}
where $\mu_0$ is the distribution measure of random variables of the form $U = Q \delta_{Y}$. The proof follows that of \cite[Proposition 5.3.1]{linde_1986} for the compound Poisson process. This, in particular, implies that the measure $\mu$ defining the Poisson point process is an \textit{infinitely divisible measure}; that is, for each $n \in \mathbb{N}$, there exists a Radon probability measure $\mu_{1/n}$ such that
\begin{align*}
\widehat{\mu}(f) = (\widehat{\mu_{1/n}}(f))^n,\quad \text{ for all }f \in \C.
\end{align*}

In fact, since $\M$ is a complete locally convex space, we have the following representation theorem from \cite[Satz 2.2]{dettweiler_1976}, which characterizes all infinitely divisible measures on $\M$. This makes use of the concept of a \Truong{L\'{e}vy} measure on a topological space; details are provided in \cite{dettweiler_1976}. The representation theorem indeed forms the foundation for constructing our class of random measures.

\begin{theorem}[(L\'{e}vy-Khintchine representation theorem)] A probability measure on $\Mm$ is infinitely divisible if and only if there exist a $u_0 \in \Mm$, a covariance operator $\mathcal{R}: \C \to \M$ and a L\'{e}vy measure $\nu$ such that
\begin{equation}
	\widehat{\mu}(f) = \exp\left[i \pair{u_0, f} - \frac{1}{2} \pair{\mathcal{R} f,f} + \int_{\Mm} (\exp(i \pair{u, f}) - 1 - i \pair{u, f}\bld{1}_F(u) )\d \nu(u)  \right],
\end{equation}
for all $f \in \C$. Here, $F$ is a convex, compact, and balanced neighborhood of $0$ (meaning that $\lambda F \subset F$ for all $\lambda$ with $|\lambda| < 1$) such that $\nu(F^{c}) < \infty.$
\end{theorem} 

We remark that the infinitely divisible property for measures has also been studied in \cite{hosseini_2017}, for Radon measures on Banach spaces.

\section{Well-posedness of the Bayesian inverse problem}\label{sec:5_wellposedness} 
Using the prior measure introduced in Section~\ref{sec:5_priordistribution}, we are ready to prove the well-posedness of the Bayesian inverse problem \eqref{eq:5_BIP}. We recall that for a separable Banach space $Y$, our goal is to determine the posterior measure $\mu_{\post}^z$ given by
\begin{align} \label{eq:radon_nikodym}
	\dfrac{\d \mu^z_{\post}}{\d \mu_{\pr}}(u) = \frac{L(z|u)}{Z(z)}, \quad \text{where } Z(z) := \int_{\Mm} L(z|u) \, \d\mu_{\pr}(u), \
\end{align}
where we have $L(z|u) := \exp(\Psi(u;z))$. Under certain assumptions on the likelihood function $L$,  we establish the well-posedness of the Bayesian inverse problem \eqref{eq:5_BIP}. In order to prove well-posedness, we make use of the Hellinger distance between probability measures, which is commonly used in comparing probability distributions, especially in the context of Bayesian inverse problems; cf. \cite{stuart_2010}. Here, we recall its definition for the sake of convenience: The Hellinger distance between two probability measures $\mu$ and $\mu'$ on $\M$, denoted by $d_{\Hell}$, is given by
\begin{align*}
	d_{\Hell} (\mu, \mu')^2 = \frac{1}{2} \int_{\M} \left( \sqrt{\dfrac{\d \mu}{\d \nu}} - \sqrt{\dfrac{\d \mu'}{\d \nu}}\right)^2 \d \nu,
\end{align*}
where both $\mu$ and $\mu'$ are absolutely continuous with respect to $\nu$.

\subsection{Well-posedness of the Bayesian inverse problem}
In what follows, we adopt the approach of \cite{latz_2020} to present the assumptions that guarantee the well-posedness of the problem.
\begin{assumption}\label{ass:wellposedS}
	We assume that the likelihood function satisfies the following conditions:
	\begin{enumerate}[label=  \normalfont(A\arabic{enumi}),ref=A\arabic{enumi}, leftmargin=1.5cm]
		\item \label{eq:A1} For almost every $u \in \M$, the map $L(\cdot|u)$ is strictly positive.
		\item  \label{eq:A2} For every $z \in Y$, $L(z|\cdot)$ is measurable and $L(z|\cdot) \in L^1(\M, \d \mu_{\pr})$.
		\item  \label{eq:A3} There exists $g \in L^1(\M, \d \mu_{\pr})$ such that $L(z|\cdot) \le g$, for every $z \in Y$.
		\item \label{eq:A4} For every $u \in \M$, the function $L(u| \cdot) : Y \to \R$ is continuous.
	\end{enumerate}
\end{assumption}

Under the given assumptions, we provide a general well-posedness result:

\begin{theorem}\label{theo:5_wellposedness} Let the assumptions in Assumption \ref{ass:wellposedS} hold. Then the Bayesian inverse problem is well-posed, in the sense that:
	\begin{enumerate}
		\item[(1)] Existence and uniqueness: For every $z \in Y$, the posterior measure $\mu_{\post}^z$ exists uniquely.
		\item[(2)] Stability: For every $z \in Y$ and every sequence $\{z_n\}_{n \in \mathbb{N}} \subset Y$ such that $z_n \to z$ in $Y$, there holds $d_{\Hell} (\mu^{z_n}_{\post}, \mu^z_{\post}) \to 0$.
	\end{enumerate}
\end{theorem}

\begin{proof}Our proof adapts that of \cite[Theorem 2.5]{latz_2020}. Firstly, let $z \in Y$ be fixed. We prove that $Z(z) > 0$. Indeed, since $L(z|\cdot ) > 0$ on $\M$ by \eqref{eq:A1}, we have 
\[\M = \bigcup_{n=1}^{\infty} \M_n \text{ where } \M_n = \left\{ u \in \M: L(z|u) \ge \frac{1}{n} \right\}, n \in \mathbb{N}.\]
As $\Mm_{n} \subset \Mm_{n+1}$, for all $n \in \mathbb{N}$, we use the $\sigma$-continuity of $\mu_{pr}$ to have
\[ \lim_{n \to \infty} \mu_{\pr}(\Mm_n) = \mu_{\pr}\left( \cup_{n=1}^{\infty} \Mm_n \right) = \mu_{\pr}(\Mm) = 1.
\]
In particular, there exists $n_0 \in \mathbb{N}$ such that $\mu_{\pr}(\Mm_{n_0}) > 0$. Hence,
\begin{align*}
	Z(z) = \int_{\Mm} L(z|u) \d \mu_{\pr}(u) > \int_{\Mm_{n_0} } L(z|u) \d \mu_{\pr}(u) \ge \frac{\mu_{\pr}(\Mm_{n_0})}{n_0} > 0.
\end{align*}
Using the Bayes' Theorem for Radon spaces, see \cite[Lemma 2.4]{latz_2020}, we obtain the unique existence of $\mu_{\post}^z$ satisfying \eqref{eq:radon_nikodym}. 

To prove the stability property (2), we first prove the continuity of the function $z \mapsto Z(z)$. Indeed, for every sequence $\{z_k\}_{k \in \mathbb{N}}$, one has $L(u|z_k) \to L(u|z)$ for almost every $u \in \M$, by assumption \eqref{eq:A4}. Hence, by \eqref{eq:A3} and the dominated convergence theorem, we obtain
\begin{align*}
	Z(z_k) = \int_{\M} L(u|z_k) \d \mu_{\pr}(u) \to \int_{\M} L(u|z) \d \mu_{\pr}(u) = Z(z).
\end{align*}
Now, by the first part, the measure $\mu_{\post}^{z_k}$ is well-defined for every $k \in \mathbb{N}$. Hence we can write
\begin{align*}
	2d_{\Hell} (\mu^{z},\mu^{z_k})^2 = \int_{\Mm} \left| \sqrt{\dfrac{L(z|u)}{Z(z)}} - \sqrt{\dfrac{L(z_k|u)}{Z(z_k)}} \right|^2 \d\mu_{\pr}(u).
\end{align*}
By the continuity of $L(\cdot|u)$ and $Z$ for almost every $u \in \M$, we have
\begin{align*}
\sqrt{\dfrac{L(z|u)}{Z(z)}} - \sqrt{\dfrac{L(z_k|u)}{Z(z_k)}} \to 0 \text{ as }k \to \infty, \text{ for a.e. } u \in \M.
\end{align*}
On the other hand, we use the fact that $(\sqrt{a} - \sqrt{b})^2 \le a + b, \forall a, b \ge 0$ to obtain
\begin{align*}
	\left|\sqrt{\dfrac{L(z|u)}{Z(z)}} - \sqrt{\dfrac{L(z_k|u)}{Z(z_k)}} \right|^2 
	&\le \dfrac{L(z|u)}{Z(z)} +\dfrac{L(z_k|u)}{Z(z_k)} \\
	&\le \frac{2}{Z(z)} (L(z|u)+ L(z_k|u)) \le \frac{4}{Z(z)} g(u),
\end{align*}
where we have used \eqref{eq:A2} and the continuity of $Z$ in the last inequality. Since $4g(\cdot)/Z(z) \in L^1(\M, \d \mu_{\pr})$ for every $z \in Y$, we again use the dominated convergence theorem to conclude that $d_{\Hell} (\mu^z, \mu^{z_k}) \to 0$. The proof is complete.
\end{proof}

\subsection{Approximation of the posterior distribution} In the previous section, we have derived the well-posedness of the Bayesian inverse problem \eqref{eq:5_BIP}. Nevertheless, in practical applications, solving the inverse problem directly in an infinite-dimensional Banach space is not feasible. Hence, approximations are necessary, and it becomes important to study whether the perturbed posterior, arising from the approximation of the forward model, converges to the posterior associated with the exact model. To formalize this, let $L_N$ denote an approximation of the likelihood $L$, obtained, for instance, through discretization of the forward operator or the underlying space. For every fixed $z \in Y$, under certain assumptions on the the approximation $L_N$, the existence of the posterior measure corresponding to the likelihood $L_N$ is guaaranted, which is denoted by $\mu^z_{\post, N}$. Formally, it is given by
\begin{equation}\label{eq:approx_bayesian}
	\dfrac{\d \mu_{\post, N}^z}{\d \mu_{\pr}} = \dfrac{L_N(z|u)}{Z_N(z)} ,\quad \text{ where } Z_N(y) = \int_{\M}  L_N(z|u)d\mu_{\pr} (u).
\end{equation}
The question is whether the sequence of  measures $\{ \mu^z_{\post, N}\}_{N \in \mathbb{N}}$ converges to $\mu^z_{\post}$, under appropriate conditions, is addressed in the following.

\begin{theorem}\label{theo:5_consistency} Assume that $L_N, L$ satisfy Assumption~\ref{ass:wellposedS} for every $N \in \mathbb{N}$, where the same upper bound $g \in L^1(\Mm, \d \mu_{\pr})$ as in \eqref{eq:A3} is suposed to hold for all $N \in \mathbb{N}$. If
\begin{align*}
	\left| L_N(z|\cdot) - L(z|\cdot) \right| \to 0 \quad \text{ almost surely in }\M,\text{ for every }z \in Y,
\end{align*}
then $\quad d_{\Hell}(\mu^z_{\post}, \mu^z_{\post, N}) \to 0$.
\end{theorem}

\begin{proof}The proof proceeds analogously to that of Theorem~\ref{theo:5_wellposedness}. Let $z \in Y$ be fixed. For every $N \in \N$, since $L_N(z|u) \to L(z|u)$ for a.e. $u \in \Mm$ and $L_N(z|u) \le g(u)$ for some $g \in L^1(\Mm, \d\mu_{\pr})$, we again use the dominated convergence theorem to have
\begin{align*}
	Z_N(z) = \int_{\Mm} L_N(z|u) \d \mu_{\pr}(u) \to \int_{\Mm} L(z|u) \d \mu_{\pr}(u) = Z(z),\quad \text{ for all }z \in Y.
\end{align*}
Hence, one has
\begin{align*}
	\sqrt{\dfrac{L_N(z|u)}{Z_N(z)}} - \sqrt{\dfrac{L(z|u)}{Z(z)}} \to 0,\quad \text{ for all } z \in Y, \text{ for a.e. } u \in \Mm.
\end{align*}
In addition, since
\begin{align*}
	\left|\sqrt{\dfrac{L_N(z|u)}{Z_N(z)}} - \sqrt{\dfrac{L(z|u)}{Z(z)}}\right|^2
	&\le \dfrac{L_N(z|u)}{Z_N(z)} + \dfrac{L(z|u)}{Z(z)}  \\
	&\le \frac{2}{Z(z)} \left( L_N(z|u) + L(z|u) \right) \le \frac{4 g(u)}{Z(z)}.
\end{align*}
Since $g \in L^1(X, \d \mu_{\pr})$, we again use the dominated convergence theorem to conclude that
\begin{align*}
	2d_{\Hell} (\mu^{z}_{\post},\mu^{z}_{\post, N})^2 = \int_{\M} \left| \sqrt{\dfrac{L(z|u)}{Z(z)}} - \sqrt{\dfrac{L_N(z|u)}{Z_N(z)}} \right|^2 \d\mu_{\pr}(u) \to 0,
\end{align*}
from which the proof is complete.
\end{proof}

\begin{remark}In Theorem~\ref{theo:5_wellposedness}, which concerns the well-posedness of the Bayesian inverse problem, and Theorem~\ref{theo:5_consistency}, which addresses the consistency of the approximations, our convergence results are provided without a convergence rate, which follows from the general setting for Bayesian inverse problems given in \cite{latz_2020}. We remark that further assumptions could be considered in order to obtain local Lipschitz continuity. For instance, the results in \cite{hosseini_2017} could be applied in our setting, by noting that the map $u \mapsto \norm{u}_{\M}$ is measurable. A detailed treatment is beyond the scope of this paper and will be the subject of future research.
\end{remark}

\section{Applications} 

\subsection{Inverse problems with Gaussian noise} Finally, to illustrate the theory, we consider some examples applicable within this framework. As is typical in parameter identification problems, we consider the problem \eqref{eq:5_BIP}, where the observation space $Y$ is finite-dimensional, i.e. $Y = \R^{N_o}$ and $\xi$ follows a Gaussian distribution, i.e. $\xi \sim \mathcal{N}(0, \Sigma)$ where $\Sigma$ is a positive-definite matrix.  In this setting, the likelihood potential function reads as
\[ \Psi(u;z) = \dfrac{1}{2}\norm{\mathcal{G}(u) - z}_{\Sigma}^2
				   = \dfrac{1}{2}\norm{\Sigma^{-1/2}(\mathcal{G}(u) - z)}^2_2.
				   \]
and $L(z|u) = \exp(-\Psi(u;z))$. As is typical, we assume that the forward operator $\G: \M(\Omega, H) \to \R^{N_o}$ is continuous in the weak* topology. Hence, one could verify that the assumptions on the likelihood function $L(z|u)$ are satisfied.

\begin{corollary}\label{cor:5_linear_gaussian}

Assume that the operator $\G: \M(\Omega, H) \to \R^{N_o}$ is weak* continuous. Then the assumptions in Assumption~\ref{ass:wellposedS} hold. 
\end{corollary}

\begin{proof}Continuity of the forward operator implies the measurability of the likelihood function. By the definition, the function $L$ is strictly positive for every $u \in \M$ and $z \in Y$. It is bounded by $1$ and the likelihood function is continuous in $z$ for any $u \in \M$. The proof is complete.
\end{proof}

\subsection{Examples} Finally, we introduce some concrete examples that can be studied in the Bayesian framework. \label{sec:example_Bayesian_point}

\begin{example}[Convolution problems with Gaussian kernels]
Let $\Omega \subset \R^d, d \ge 1$ be a compact domain with a non-empty interior. We consider the Gaussian kernel depending on $\sigma$ $k = k_\sigma: \Omega \times \Omega \to \R$ given by
\begin{align*}
    k(x,y) = k_{\sigma}(x,y) := \exp \left(-\dfrac{|x-y|^2}{2\sigma^2}\right), \quad x, y \in \Omega.
\end{align*}
Let the signal to identify be a real-valued discrete measure, represented as
\begin{align*} u = \sum_{k=1}^{N_s} q_k \delta_{y_k},\quad q_k \in \mathbb{R},\ y_k \in \Omega. 
\end{align*}

On $\Omega$, we fix a finite set of measurement locations $\bld{x} = (x_1,\ldots, x_{N_o})$ and consider the vector kernel
\begin{align*}
    k[\bld{x}, y] := (k(x_1, y); k(x_2, y);\ldots; k(x_{N_o}, y)), \quad y \in \Omega,
\end{align*}
as well as the (linear) forward operator $\mathcal{G}: \M(\Omega) \to \R^{N_o}$ defined by
\[ \mathcal{G}\miu = \int_{\Omega} k[\bld{x}, y] \d \miu(y),\quad \miu \in \M(\Omega). \]
Together with the additive noise $\xi$, we aim to determine $u$ from the measurement data $z$ through the model $z = \G u + \xi$. Some applications of this problem have been studied in, for instance, \cite{mccutchen_1967, Candes_Fernandez-Granda_2013}. In addition, the recent work \cite{huynh2023optimal} addresses the problem of selecting optimal sensor placements for this setting.
    
As the kernel is continuous, the map $\mathcal{G} : \M(\Omega) \to \R^{N_o}$ is weak* continuous. By Corollary~\ref{cor:5_linear_gaussian}, Theorem~\ref{theo:5_wellposedness} is applicable and the Bayesian inverse problem in this setting is well-posed. We define the prior distribution of $u$ through the random measure \begin{align*}
		u = \sum_{k=1}^K Q_k \delta_{Y_k},
	\end{align*}
	where we consider the distributions $K \sim \text{Poiss}(\gamma)$, $Q_k \sim \mathcal{N}(\mu, \sigma^2)$ and $Y_k \sim \text{Uniform}(\Omega)$. If we know that $q_k$ are positive coefficients, for instance in biological imaging \cite{schiebinger_robeva_recht_2018}, one could consider the Log-normal distribution, $Q_k \sim \text{LogNormal}(\mu,\sigma^2)$, which is known to be supported on $(0,\infty)$.
\end{example}

\begin{example}[(Sound source localization with Helmholtz equation)] The inverse sound source location problem seeks to recover an unknown acoustic source $u$, modeled as a superposition of time-harmonic monopoles, from noisy pointwise measurements of the acoustic pressure, that is $u$ has the form
	\begin{align*}
		u = \sum_{k=1}^{N_s} q_k \delta_{y_k}, \quad q_k \in \mathbb{C}, y_k \in \Omega_s,
	\end{align*}
	where $\Omega_s$ denotes the source domain. The problem is governed by the Helmholtz equation on a bounded domain; details can be found in \cite{pieper_tang_trautmann_walter_2020}. 
	
	In this setting, by \cite[Lemma 2.4]{pieper_tang_trautmann_walter_2020}, the solution operator  $S: \M(\Omega_S, \mathbb{C}) \to \C(\Omega_O, \mathbb{C})$ is linear and bounded, where $\Omega_O$ denotes the observation set with $\overline{\Omega_O} \cap \overline{\Omega_S} = \varnothing$, which implies that the (linear) observation operator $\G: \M(\Omega_S, \mathbb{C}) \to \mathbb{C}^{N_o}$ given by 
	\begin{align*}
		\G u = \left(S[u](x_1), \ldots, S[u](x_{N_o})\right), \quad x_1, \ldots, x_{N_o} \in \Omega_o
	\end{align*}
	is weak*-to-strong continuous. Hence, Theorem~\ref{theo:5_wellposedness} can be applied and the Bayesian inverse problem is well-posed. 
    
    We remark that the Bayesian inverse problem in this setting has also been studied in \cite{engel_hafemeyer_munch_schaden_2019}, where the prior is defined on $\ell_1$ via a sequence of measures $\{ \mu_{\pr, k}\}_{k \in \mathbb{N}}$. In our setting, the prior measure is naturally defined through \eqref{eq:radon_random_measure}. Here, one could consider again that $K \sim \text{Poiss}(\lambda)$, $Q_k \sim \text{Complex}\mathcal{N}(q, \sigma^2, c^2)$, where $\text{Complex}\mathcal{N}$ is the complex Gaussian distribution, and $Y_k \sim \text{Uniform}(\Omega_S)$, for $k \in \mathbb{N}$.
\end{example}

\section{Conclusion and remarks}

In this work, we study the Bayesian inverse problem \eqref{eq:5_BIP}, where the parameter to be identified belongs to a space of measures and, as such, inherits its sparse structure from the ambient space. To define a prior distribution in this space, we consider an appropriate topological structure--namely, the weak* topology--together with its corresponding Borel $\sigma$-algebra. The priors are characterized via point processes, which are measurable with respect to the underlying structure. With the given priors, we establish the well-posedness of the Bayesian inverse problem, as well as the consistency under the approximation of the likelihood function and the prior measure.

Nevertheless, a numerical study should be conducted to demonstrate the practical applicability of the approach. In addition, it is evident that appropriate choices of point processes are essential to ensure accurate reconstruction. These topics will be addressed in future work.

\section*{Acknowledgement}
This work was supported by the Austrian Science Fund (FWF) under the grant DOC78. The author owes a debt of gratitude to Prof. Barbara Kaltenbacher for her valuable discussions and comments, which significantly contributed to the improvement of this manuscript. He is also grateful to Prof. Daniel Walter for his comments and for his interest in a forthcoming collaboration that builds on this work.

\bibliography{OED}
\bibliographystyle{abbrv}
\end{document}